\documentclass[12 pt]{amsart}
\usepackage{amsmath,times,epsfig,amssymb,amsbsy,amscd,amsfonts,amstext,color,bm,comment}
\usepackage[arrow,matrix]{xy}

\usepackage{here}

\usepackage{tikz}

\numberwithin{equation}{section}

\theoremstyle{plain}
\newtheorem{theorem}{Theorem}[section]

\newtheorem{lemma}[theorem]{Lemma}

\newtheorem{proposition}[theorem]{Proposition}

\theoremstyle{definition}
\newtheorem{definition}[theorem]{Definition}
\newtheorem{example}[theorem]{Example}

\theoremstyle{remark}
\newtheorem{remark}[theorem]{Remark}

\newcommand{\Z}{\mathbb{Z}}

\newcommand{\R}{\mathbb{R}}

\newcommand{\scL}{\mathcal{L}}

\newcommand{\conv}{\mathop{\rm Conv}\nolimits}
\newcommand{\overphi}{\varphi}

\begin{document}

\title[Chip-firing games]{Finite record sets of chip-firing games}

\begin{abstract}
A finite graph with an assignment of non-negative integers to 
vertices gives chip-firing games. 
Chip-firing games determine languages (sets of words) called 
the record sets of legal games. 
Bj\"orner, Lov\'asz and Shor found several properties that are satisfied 
by record sets. In this paper, we will find two more properties of record sets. 
Under the assumption that the record set is finite and 
the game fires only two vertices, 
these properties characterize the record sets of graphs. 
\end{abstract}

\author{Kentaro Akasaka}
\address{Kentaro Akasaka: 
Department of Mathematics, Faculty of Science, Hokkaido University, 
Kita 10, Nishi 8, Kita-Ku, Sapporo 060-0810, Japan.}
\email{ken.akasaka.japan@gmail.com}

\author{Suguru Ishibashi}
\address{Suguru Ishibashi: 
Department of Mathematics, Faculty of Science, Hokkaido University, 
Kita 10, Nishi 8, Kita-Ku, Sapporo 060-0810, Japan.}
\email{sgr.ishibashi@gmail.com}

\author{Masahiko Yoshinaga}
\address{Masahiko Yoshinaga: 
Department of Mathematics, Osaka University, 
Toyonaka, Osaka 560-0043, Japan}
\email{yoshinaga@math.sci.osaka-u.ac.jp}

\thanks{The authors thank Ahmed Umer Ashraf and Tan Nhat Tran 
for fruitful discussion on chip-firing games. The authors also thank 
the referee(s) for careful reading and lot of suggestions to improve the paper. 
M. Yoshinaga is partially supported by 
JSPS KAKENHI Grant Numbers JP18H01115, JP19K21826}


\date{\today}
\maketitle



\section{Introduction}
\label{sec:intro}
The chip-firing game is a one-player game played on a graph. 
The vertices of the graph have multiple chips,
the player chooses a vertex whose degree is equal to or less than the
number of chips. 
Then, the operation ``firing a vertex'' 
moves 
chips on the vertex to adjacent vertices along the edges 
(see \S \ref{sec:notation} for precise formulation). 
Repeat this operation until the number of chips at each vertex is less than the
degree.

There are various aspects of research on chip-firing games, see 
\cite{klivans} for details. 
As one of the techniques for studying chip-firing games, 
Bj\"orner, Lov\'asz and Shor 
constructed a formal language (a set of words) from chip-firing games \cite{bls}.
By studying these formal language, they proved basic results on chip-firing games.

The purpose of this paper is to examine the relationship between 
chip-firing games and formal languages in more detail.
In particular, we aim to obtain the necessary and sufficient conditions for a
given finite language to be obtained from chip-firing games.

\section{Notation and background}
\label{sec:notation}

Let $G$ be a finite connected graph 
(without loops) on the vertex set $V=\{1, \dots, n\}$. 
Recall that $\deg(i)$ is the number of edges adjacent to $i$. 
For two vertices $i, j\in V$, denote by $e(i, j)$ the number of 
edges connecting $i$ and $j$. 
Let $(\varphi_i)_{i\in V}\in\Z_{\geq 0}^n$ be a vector with nonnegative 
integer components. We consider $(\varphi_i)_{i\in V}$ as a configuration 
of chips. 
Recall that firing the vertex $i\in V$ means that we send a chip on $i$ along 
each edge adjacent to $i$ to the opposite vertex. Then we obtain a 
new configuration $(\varphi'_i)_{i\in V}$, which is 
\[
\varphi'_j=
\left\{
\begin{array}{cc}
\varphi_i-\deg(i),& j=i\\
\varphi_j+e(i, j),& j\neq i. 
\end{array}
\right.
\]
The firing is called {\bf legal} if $(\varphi'_i)_{i\in V}$ is a nonnegative 
integer vector. A sequence of legal firings is called a {\bf legal game}. 
A vertex $i\in V$ is said to be ready if $\varphi_i\geq \deg(i)$. 
The game terminates if there are no vertices ready to fire. 

We say that the game is finite if it terminates after finitely many firings. Conversely, if it does not terminate in arbitrarily many firings, we say the game is infinite. Whether a game is finite or not depends on the graph and the initial chip configuration. 

Next, we briefly recall several notions on words and languages. 
A {\bf word} on the alphabet 
$\Sigma=\{1, \dots, n\}$ is a  finite string of elements of $\Sigma$.  
We denote the empty word by $\epsilon$. 
The length of a word $w$ is denoted by $|w|$. 
A word $u$ is called a {\bf beginning section} of $w$ if 
$w$ is expressed as $w=uv$ by a word $v$. 
A {\bf subword} of a word $w$ is obtained by deleting letters from $w$ arbitrarily. 
(Note that a subword need not consist of consecutive letters of the word $w$.) 
The {\bf score} $[w]$ is a vector in $\Z_{\geq 0}^n$ whose $i$-th component 
is the number of the letter $i$ appearing in $w$. We denote the coordinate-wise 
maximum of $[w_1]$ and $[w_2]$ by $[w_1]\vee [w_2]\in\Z_{\geq 0}^n$. 

\begin{definition}
For a game on the graph $G$ (with vertex set $V$) 
and a configuration $(\varphi_i)_{i\in V}$, 
we make a word on $V$ by arranging the symbols of the vertices 
in the order of firing in the game. 
We call this word the {\bf record} of the game. 
We can construct a language $L_{G,\varphi}$ on $V$ as the set of records of
all legal games on the graph $G$ and the configuration $(\varphi_i)_{i\in V}$. 
We call this language the {\bf record set} for $G$ and $(\varphi_i)_{i\in V}$. 
\end{definition}

It has been shown in \cite{bls} that the record set has the following properties. 

\begin{definition}
Let $\scL$ be a language on the alphabet $\Sigma$. 
\begin{itemize}
\item[(LH)] 
We say that $\scL$ is {\bf left-hereditary} if any beginning section of every word 
$w\in \scL$ also belongs to $\scL$. 
\item[(LF)] 
We say that $\scL$ is {\bf locally free} if: 
Let $w \in \scL$ and $a, b\in\Sigma$ with $a\neq b$. 
If $wa, wb \in \scL$, then $wab \in \scL$. 
\item[(PM)] 
We say that $\scL$ is {\bf permutable} if: 
Let $u, w \in \scL$ with $[u] = [w]$.  If $ua \in \scL$ for some $a \in \Sigma$, 
then $wa \in \scL$. 
\item[(SE)] 
We say that $\scL$ has {\bf strong exchange property} if: 
If $u, v \in \scL$ then $u$ contains a subword $u'$ such that $wu' \in \scL$ and $[wu'] = [u] \vee [w]$
\end{itemize}
\end{definition}

\begin{proposition}[\cite{bls}]
\begin{itemize}
\item[(1)] 
If the language $\scL$ satisfies (LH), (LF), (PM), then it also satisfies (SE). 
\item[(2)] 
The set of records $L_{G, \varphi}$ for a finite graph $G$ and an 
initial configuration $\varphi=(\varphi_i)_{i\in V}$ satisfies (LH), (LF), (PM). 
\item[(3)] 
If $L_{G, \varphi}$ is finite, then there is at least one vertex that is never fired.  
(In this case, let us denote by $\Sigma$ ($\subsetneq V$) the set of 
all fired vertices.)  
\end{itemize}
\end{proposition}

We observe that the above properties do not characterize the record sets. 
Indeed, there is a language which satisfies these properties, but can not 
be expressed in the form $L_{G, \varphi}$. (See Lemma \ref{ABB} for details.) 
\begin{example}\label{CE}
Let $\scL = \{ \epsilon, 1, 12, 122\}$. 
Then $\scL$ satisfies (LH), (LF), and (PM), but cannot be a record set for a chip-firing game. 
\end{example}

\section{Results}

\subsection{abb-property.}

To characterize the language defined as record sets of chip-firing games, 
we have to exclude languages as in Example \ref{CE}. 

\begin{definition}
Let $\scL$ be a language on the alphabets $\Sigma$. 
\begin{itemize}
\item[(abb)] 
We say that $\scL$ has {\bf abb-property} (or satisfies (abb)) if: Let 
$w\in\scL$ and $a, b\in\Sigma$, if $wabb\in\scL$ then $wb\in\scL$. 
\end{itemize}
\end{definition}

We will show that the record set satisfies abb property. 
\begin{lemma}\label{ABB}
Given a graph $G$ with configuration of chips $\varphi$, 
the record set $L_{G, \varphi}$ of legal games satisfies (abb). 
\end{lemma}
\begin{proof}
Let $(\varphi_i)_{i\in V}$ be the chip configuration 
after finishing a legal game $w$. 
When $wabb$ is a legal game for different symbols $a$ and $b$,
\[
	\varphi_b + e(a, b) \geq 2 \deg(b). 
\]
Since $e(a, b) \leq \deg(b)$, 
\[
	\varphi_b \geq 2 \deg(b) - e(a, b) \geq 2\deg(b) - \deg(b) = \deg(b)
\]
This implies $wb$ is legal. 
\end{proof}

\subsection{Separation property}

Let $G$ be a graph with vertex set $V$. 
Let $\varphi$ be an initial configuration of chips. As we have already 
mentioned, if the record set $L_{G, \varphi}$ is finite, then the 
set of fired vertices is a proper subset $\Sigma\subsetneq V$. 
Suppose $\Sigma=\{1, \dots, n\}$. 
Let 
$f_i\in\Z^n$ be an integral vector whose $i$-th component is 
$-\deg(i)$ and $j$-th component ($j\neq i$) is $e(i, j)$. Then 
the firing $i\in\Sigma$ is equivalent to the vector $f_i$ is added to $\varphi$. 
Then legal games can be interpreted as sequences of lattice points 
in the quadrant $D:=\Z_{\geq 0}^n$. 

\begin{example}
Suppose $\Sigma=\{1, 2\}$. 
Then $f_1 = \left( \begin{array}{c} -\deg(1) \\ e(1,2) \end{array} \right)$
and $f_2 = \left( \begin{array}{c} e(1,2) \\ -\deg(2) \end{array} \right)$. 
A word $w = i_1i_2 \dots i_m$ on $\Sigma = \{1,2\}$ ($i_p\in\{1,2\}$) 
is the record of a legal game 
if and only if $\overphi+f_{i_1}+\dots+f_{i_p}\in D$ ($p=1,\dots,m$) 
(See Figure \ref{1game}).  
\begin{figure}[H]
\begin{center}
\begin{tikzpicture}
	\draw[->, >=stealth, very thick] (-0.5,0) -- (4, 0)node[above]{$x_1$};
	\draw[->, >=stealth, very thick] (0,-0.5) -- (0, 4)node[right]{$x_2$};
	\draw (0, 0)node[below right] {O};
	\coordinate (p1) at (3.5, 3);
	\coordinate (p2) at (2, 3.5);
	\coordinate (p3) at (2.5, 2.5);
	\coordinate (p4) at (3, 1.5);
	\coordinate (p5) at (1.5, 2);
	\coordinate (p6) at (0, 2.5);
	\coordinate (p7) at (0.5, 1.5);
	\filldraw (p1) circle[radius = 0.5mm]; 
	\filldraw (p2) circle[radius = 0.5mm]; 
	\filldraw (p3) circle[radius = 0.5mm]; 
	\filldraw (p4) circle[radius = 0.5mm]; 
	\filldraw (p5) circle[radius = 0.5mm]; 
	\filldraw (p6) circle[radius = 0.5mm]; 
	\filldraw (p7) circle[radius = 0.5mm]; 
	\draw[->, >=stealth] (p1)node[above right]{$\overphi$} -- (2.03, 3.49);
	\draw[->, >=stealth] (p2) -- (2.485, 2.53);
	\draw[->, >=stealth] (p3) -- (2.985, 1.53);
	\draw[->, >=stealth] (p4) -- (1.53, 1.99);
	\draw[->, >=stealth] (p5) -- (0.03, 2.49);
	\draw[->, >=stealth] (p6) -- (0.485, 1.53);
	\node[rectangle, white, text = black, inner sep = 0.1em] (P3) at (2.76, 3.6) {$f_1$};
	\node[rectangle, white, text = black, inner sep = 0.1em] (P3) at (2, 2.9) {$f_2$};
\end{tikzpicture}
\caption{A walk corresponding to the word $w=122112$}
\label{1game}
\end{center}
\end{figure}
\end{example}

Note that the set of all records can be described as 
the set of all possible walks in the first quadrant $D=\Z_{\geq 0}^n$ 
(Figure \ref{allgame}). 

From the combinatorial point of view, it is also important to consider 
the affine transformation $F:\R^n\longrightarrow\R^n$ which sends 
$\varphi$ to $0$, and the vector $f_i$ to the standard basis vector $e_i$ 
(Figure \ref{koushi2}). Since the move defined by the addition of $e_i$ 
corresponds to a firing at the vertex $i\in\Sigma$, 
the lattice points are corresponding to the scores. 
Furthermore, a legal game is corresponding to a walk 
on the lattice points in the quadrant $D$. 
Now we introduce the following subset of scores. 

\begin{definition}
\label{def:Xi}
Let $\scL$ be a language on $\Sigma=\{1, \dots, n\}$. Denote by 
$[\scL]:=\{[w]\in\Z_{\geq 0}^n\mid w\in\scL\}$ the set of all 
scores. Define the subset $X_i\subset[\scL]$ by 
\[
X_i:=\{[w]\in[\scL]\mid [w]+e_i\notin[\scL]\}, 
\]
and denote by $X_i':=e_i+X_i$ the translation by the vector $e_i$. 
(See Figure \ref{koushi2} and Figure \ref{koushi3}.) 
\end{definition}
From this moment, we think the map $F$ has been applied. 
From the construction, it is easy to prove the following. 
\begin{lemma}\label{SP}
Let $G$ be a graph and $\varphi$ an initial configuration. 
Suppose $L_{G, \varphi}$ is finite. (Note that $L_{G, \varphi}$ is 
a language on the set of fired vertices $\Sigma$. ) 
Then 
$L_{G, \varphi}$ satisfies the following Separation Property (SP). 
\begin{itemize}
\item[(SP)] For $i=1, \dots, n$, 
$\conv(X_i)$ and $\conv (X'_i)$ are separated by a hyperplane. 
\end{itemize}
\end{lemma}

\begin{proof}
Consider the inverse image of the map $F$. The sets 
$F^{-1}(X_i)$ and $F^{-1}(X'_i)$ are separated by the hyperplane 
$x_i=-\varepsilon$ ($0<\varepsilon\ll 1$) (Figure \ref{allgame+}). 
Then the image of the hyperplane separates 
$\conv(X_i)$ and $\conv (X'_i)$. 
\end{proof}


\begin{figure}[H]
\begin{tabular}{cc}
\begin{minipage}{0.5\hsize}
\begin{center}
\begin{tikzpicture}
	\draw[->, >=stealth, very thick] (-0.5,0) -- (5.5, 0)node[above]{$x_1$};
	\draw[->, >=stealth, very thick] (0,-0.5) -- (0, 5.5)node[right]{$x_2$};
	\draw (0, 0)node[below right] {O};
	\coordinate (p1) at (3.5, 3);
	\coordinate (p2) at (2, 3.5);
	\coordinate (p3) at (2.5, 2.5);
	\coordinate (p4) at (3, 1.5);
	\coordinate (p5) at (1.5, 2);
	\coordinate (p6) at (0, 2.5);
	\coordinate (p7) at (0.5, 1.5);
	\coordinate (p8) at (0.5, 4);
	\coordinate (p9) at (1, 3);
	\coordinate (p10) at (4, 2);
	\coordinate (p11) at (4.5, 1);
	\coordinate (p12) at (5, 0);
	\coordinate (p13) at (3.5, 0.5);
	\coordinate (p14) at (2, 1.0);
	\coordinate (p15) at (2.5, 0);
	\coordinate (p16) at (1, 0.5);
	\coordinate (p17) at (5.5, -1);	
	\filldraw (p1) circle[radius = 0.5mm]; 
	\filldraw (p2) circle[radius = 0.5mm]; 
	\filldraw (p3) circle[radius = 0.5mm]; 
	\filldraw (p4) circle[radius = 0.5mm]; 
	\filldraw (p5) circle[radius = 0.5mm]; 
	\filldraw (p6) circle[radius = 0.5mm]; 
	\filldraw (p7) circle[radius = 0.5mm]; 
	\filldraw (p8) circle[radius = 0.5mm]; 
	\filldraw (p9) circle[radius = 0.5mm]; 
	\filldraw (p10) circle[radius = 0.5mm]; 
	\filldraw (p11) circle[radius = 0.5mm]; 
	\filldraw (p12) circle[radius = 0.5mm]; 
	\filldraw (p13) circle[radius = 0.5mm]; 
	\filldraw (p14) circle[radius = 0.5mm]; 
	\filldraw (p15) circle[radius = 0.5mm]; 
	\filldraw (p16) circle[radius = 0.5mm]; 
	\filldraw [fill = red, white](p17) circle[radius = 0.5mm]; 
	\draw[->, >=stealth] (p1)node[above right]{$\overphi$} -- (2.03, 3.49);
	\draw[->, >=stealth] (p1) -- (3.985, 2.03);
	\draw[->, >=stealth] (p2) -- (0.53, 3.99);
	\draw[->, >=stealth] (p2) -- (2.485, 2.53);
	\draw[->, >=stealth] (p3) -- (2.985, 1.53);
	\draw[->, >=stealth] (p3) -- (1.03, 2.99);
	\draw[->, >=stealth] (p4) -- (1.53, 1.99);
	\draw[->, >=stealth] (p4) -- (3.485, 0.53);
	\draw[->, >=stealth] (p5) -- (1.985, 1.03);
	\draw[->, >=stealth] (p5) -- (0.03, 2.49);
	\draw[->, >=stealth] (p6) -- (0.485, 1.53);
	\draw[->, >=stealth] (p7) -- (0.985, 0.53);
	\draw[->, >=stealth] (p8) -- (0.985, 3.03);
	\draw[->, >=stealth] (p9) -- (1.485, 2.03);
	\draw[->, >=stealth] (p10) -- (2.53, 2.49);
	\draw[->, >=stealth] (p10) -- (4.485, 1.03);
	\draw[->, >=stealth] (p11) -- (4.985, 0.03);
	\draw[->, >=stealth] (p11) -- (3.03, 1.49);
	\draw[->, >=stealth] (p12) -- (3.53, 0.49);
	\draw[->, >=stealth] (p13) -- (2.03, 0.99);
	\draw[->, >=stealth] (p14) -- (2.485, 0.03);
	\draw[->, >=stealth] (p14) -- (0.53, 1.49);
	\draw[->, >=stealth] (p15) -- (1.03, 0.49);
	\node[rectangle, white, text = black, inner sep = 0.1em] (P3) at (2.76, 3.6) {$f_1$};
	\node[rectangle, white, text = black, inner sep = 0.1em] (P3) at (4.03, 2.6) {$f_2$};
\end{tikzpicture}
\caption{Record set (before $F$)}
\label{allgame}
\end{center}
\end{minipage}

\begin{minipage}{0.5\hsize}
\begin{center}
\begin{tikzpicture}
	\draw[->, >=stealth, very thick] (-1.5,0) -- (5.5, 0)node[above]{$x_1$};
	\draw[->, >=stealth, very thick] (0,-0.5) -- (0, 5.5)node[right]{$x_2$};
	\draw (0, 0)node[below right] {O};
	\coordinate (p1) at (3.5, 3);
	\coordinate (p2) at (2, 3.5);
	\coordinate (p3) at (2.5, 2.5);
	\coordinate (p4) at (3, 1.5);
	\coordinate (p5) at (1.5, 2);
	\coordinate (p6) at (0, 2.5);
	\coordinate (p7) at (0.5, 1.5);
	\coordinate (p8) at (0.5, 4);
	\coordinate (p9) at (1, 3);
	\coordinate (p10) at (4, 2);
	\coordinate (p11) at (4.5, 1);
	\coordinate (p12) at (5, 0); 
	\coordinate (p13) at (3.5, 0.5); 
	\coordinate (p14) at (2, 1.0); 
	\coordinate (p15) at (2.5, 0); 
	\coordinate (p16) at (1, 0.5); 
	\coordinate (p17) at (-1, 4.5);
	\coordinate (p18) at (-0.5, 3.5);
	\coordinate (p19) at (-1.5, 3);
	\coordinate (p20) at (-1.0, 2);
	\coordinate (p21) at (-0.5, 1);
	\coordinate (p22) at (5.5, -1);
	\coordinate (p23) at (4, -0.5);
	\coordinate (p24) at (3, -1);
	\coordinate (p25) at (1.5, -0.5); 		
	\draw[->, >=stealth] (p1) -- (2.03, 3.49);
	\draw[->, >=stealth] (p1) -- (3.985, 2.03);
	\draw[->, >=stealth] (p2) -- (0.53, 3.99);
	\draw[->, >=stealth] (p2) -- (2.485, 2.53);
	\draw[->, >=stealth] (p3) -- (2.985, 1.53);
	\draw[->, >=stealth] (p3) -- (1.03, 2.99);
	\draw[->, >=stealth] (p4) -- (1.53, 1.99);
	\draw[->, >=stealth] (p4) -- (3.485, 0.53);
	\draw[->, >=stealth] (p5) -- (1.985, 1.03);
	\draw[->, >=stealth] (p5) -- (0.03, 2.49);
	\draw[->, >=stealth] (p6) -- (0.485, 1.53);
	\draw[->, >=stealth] (p7) -- (0.985, 0.53);
	\draw[->, >=stealth] (p8) -- (0.985, 3.03);
	\draw[->, >=stealth] (p9) -- (1.485, 2.03);
	\draw[->, >=stealth] (p10) -- (2.53, 2.49);
	\draw[->, >=stealth] (p10) -- (4.485, 1.03);
	\draw[->, >=stealth] (p11) -- (4.985, 0.03);
	\draw[->, >=stealth] (p11) -- (3.03, 1.49);
	\draw[->, >=stealth] (p12) -- (3.53, 0.49);
	\draw[->, >=stealth] (p13) -- (2.03, 0.99);
	\draw[->, >=stealth] (p14) -- (2.485, 0.03);
	\draw[->, >=stealth] (p14) -- (0.53, 1.49);
	\draw[->, >=stealth] (p15) -- (1.03, 0.49);
	\draw[->, >=stealth, red] (p8) -- (-0.97, 4.49);
	\draw[->, >=stealth, red] (p9) -- (-0.47, 3.49);
	\draw[->, >=stealth, red] (p6) -- (-1.47, 2.99);
	\draw[->, >=stealth, red] (p7) -- (-0.97, 1.99);
	\draw[->, >=stealth, red] (p16) -- (-0.485, 0.99);
	\draw[->, >=stealth, blue] (p12) -- (5.485, -0.97);
	\draw[->, >=stealth, blue] (p13) -- (3.985, -0.47);
	\draw[->, >=stealth, blue] (p15) -- (2.985, -0.97);
	\draw[->, >=stealth, blue] (p16) -- (1.485, -0.47);
	\filldraw (p1) circle[radius = 0.5mm]; 
	\filldraw (p2) circle[radius = 0.5mm]; 
	\filldraw (p3) circle[radius = 0.5mm]; 
	\filldraw (p4) circle[radius = 0.5mm]; 
	\filldraw (p5) circle[radius = 0.5mm]; 
	\filldraw (p6) circle[radius = 0.5mm]; 
	\filldraw (p7) circle[radius = 0.5mm]; 
	\filldraw (p8) circle[radius = 0.5mm]; 
	\filldraw (p9) circle[radius = 0.5mm]; 
	\filldraw (p10) circle[radius = 0.5mm]; 
	\filldraw (p11) circle[radius = 0.5mm]; 
	\filldraw (p12) circle[radius = 0.5mm]; 
	\filldraw (p13) circle[radius = 0.5mm]; 
	\filldraw (p14) circle[radius = 0.5mm]; 
	\filldraw (p15) circle[radius = 0.5mm]; 
	\filldraw (p16) circle[radius = 0.5mm]; 
	\filldraw [fill = red, red](p17) circle[radius = 0.5mm]; 
	\filldraw [fill = red, red] (p18) circle[radius = 0.5mm]; 
	\filldraw [fill = red, red] (p19) circle[radius = 0.5mm]; 
	\filldraw [fill = red, red] (p20) circle[radius = 0.5mm]; 
	\filldraw [fill = red, red] (p21) circle[radius = 0.5mm]; 
	\filldraw [fill = red, blue] (p22) circle[radius = 0.5mm]; 
	\filldraw [fill = red, blue] (p23) circle[radius = 0.5mm]; 
	\filldraw [fill = red, blue] (p24) circle[radius = 0.5mm]; 
	\filldraw [fill = red, blue] (p25) circle[radius = 0.5mm]; 
\end{tikzpicture}
\caption{}
\label{allgame+}
\end{center}
\end{minipage}
\end{tabular}
\end{figure}

\begin{figure}[H]
\begin{tabular}{cc}
\begin{minipage}{0.5\hsize}
\begin{center}
\begin{tikzpicture}
	\draw[->, >=stealth, very thick] (-0.7,0) -- (5, 0)node[above]{$x$};
	\draw[->, >=stealth, very thick] (0,-0.7) -- (0, 5.5)node[right]{$y$};
	\draw (0, 0)node[below left] {O};
	\coordinate (p1) at (0, 0);
	\coordinate (p2) at (0, 1);
	\coordinate (p3) at (0, 2);
	\coordinate (p4) at (0, 3);
	\coordinate (p5) at (1, 0);
	\coordinate (p6) at (1, 1);
	\coordinate (p7) at (1, 2);
	\coordinate (p8) at (1, 3);
	\coordinate (p9) at (2, 0);
	\coordinate (p10) at (2, 1);
	\coordinate (p11) at (2, 2);
	\coordinate (p12) at (2, 3);
	\coordinate (p13) at (2, 4);
	\coordinate (p14) at (3, 2);
	\coordinate (p15) at (3, 3);
	\coordinate (p16) at (3, 4);
	\coordinate (p17) at (3, 0);
	\coordinate (p18) at (3, 1);
	\coordinate (p19) at (4, 2);
	\coordinate (p20) at (4, 3);
	\coordinate (p21) at (4, 4);
	\coordinate (p22) at (0, 4);
	\coordinate (p23) at (1, 4);
	\coordinate (p24) at (2, 5);
	\coordinate (p25) at (3, 5);
	\draw[->, >=stealth] (p1) -- (0.98, 0); 
	\draw[->, >=stealth] (p2) -- (0.98, 1);
	\draw[->, >=stealth] (p3) -- (0.98, 2);
	\draw[->, >=stealth] (p4) -- (0.98, 3);
	\draw[->, >=stealth] (p5) -- (1.98, 0);
	\draw[->, >=stealth] (p6) -- (1.98, 1);
	\draw[->, >=stealth] (p7) -- (1.98, 2);
	\draw[->, >=stealth] (p8) -- (1.98, 3);
	\draw[->, >=stealth] (p11) -- (2.98, 2);
	\draw[->, >=stealth] (p12) -- (2.98, 3);
	\draw[->, >=stealth] (p13) -- (2.98, 4);
	\draw[->, >=stealth, red] (p9) -- (2.98, 0);
	\draw[->, >=stealth, red] (p10) -- (2.98, 1);
	\draw[->, >=stealth, red] (p14) -- (3.98, 2);
	\draw[->, >=stealth, red] (p15) -- (3.98, 3);
	\draw[->, >=stealth, red] (p16) -- (3.98, 4);
	\draw[->, >=stealth] (p1) -- (0, 0.98); 
	\draw[->, >=stealth] (p2) -- (0, 1.98);
	\draw[->, >=stealth] (p3) -- (0, 2.98);
	\draw[->, >=stealth] (p5) -- (1, 0.98);
	\draw[->, >=stealth] (p6) -- (1, 1.98);
	\draw[->, >=stealth] (p7) -- (1, 2.98);
	\draw[->, >=stealth] (p9) -- (2, 0.98);
	\draw[->, >=stealth] (p10) -- (2, 1.98);
	\draw[->, >=stealth] (p11) -- (2, 2.98);
	\draw[->, >=stealth] (p12) -- (2, 3.98);
	\draw[->, >=stealth] (p14) -- (3, 2.98);
	\draw[->, >=stealth] (p15) -- (3, 3.98);
	\draw[->, >=stealth, blue] (p4) -- (0, 3.98);
	\draw[->, >=stealth, blue] (p8) -- (1, 3.98);
	\draw[->, >=stealth, blue] (p13) -- (2, 4.98);
	\draw[->, >=stealth, blue] (p16) -- (3, 4.98);
	\filldraw (p1) circle[radius = 0.5mm]; 
	\filldraw (p2) circle[radius = 0.5mm]; 
	\filldraw (p3) circle[radius = 0.5mm]; 
	\filldraw (p4) circle[radius = 0.5mm]; 
	\filldraw (p5) circle[radius = 0.5mm]; 
	\filldraw (p6) circle[radius = 0.5mm]; 
	\filldraw (p7) circle[radius = 0.5mm]; 
	\filldraw (p8) circle[radius = 0.5mm]; 
	\filldraw (p9) circle[radius = 0.5mm]; 
	\filldraw (p10) circle[radius = 0.5mm]; 
	\filldraw (p11) circle[radius = 0.5mm]; 
	\filldraw (p12) circle[radius = 0.5mm]; 
	\filldraw (p13) circle[radius = 0.5mm]; 
	\filldraw (p14) circle[radius = 0.5mm]; 
	\filldraw (p15) circle[radius = 0.5mm]; 
	\filldraw (p16) circle[radius = 0.5mm]; 
	\filldraw [fill = red, red](p17) circle[radius = 0.5mm]; 
	\filldraw [fill = red, red] (p18) circle[radius = 0.5mm]; 
	\filldraw [fill = red, red] (p19) circle[radius = 0.5mm]; 
	\filldraw [fill = red, red] (p20) circle[radius = 0.5mm]; 
	\filldraw [fill = red, red] (p21) circle[radius = 0.5mm]; 
	\filldraw [fill = red, blue] (p22) circle[radius = 0.5mm]; 
	\filldraw [fill = red, blue] (p23) circle[radius = 0.5mm]; 
	\filldraw [fill = red, blue] (p24) circle[radius = 0.5mm]; 
	\filldraw [fill = red, blue] (p25) circle[radius = 0.5mm]; 
	\draw[domain=2.1:4.2, thick] plot(\x, 3*\x -7);
	\draw[domain=-0.7:5.1, thick] plot(\x, 0.5*\x +3);
	\node[rectangle, white, text = black, inner sep = 0.1em] (P1) at (2.6, -0.5) {$l_1$};
	\node[rectangle, white, text = black, inner sep = 0.1em] (P1) at (-0.7, 3.1) {$l_2$};
\end{tikzpicture}
\caption{$X_i$ (after $F$)}
\label{koushi2}
\end{center}
\end{minipage}

\begin{minipage}{0.5\hsize}
\begin{center}
\begin{tikzpicture}
	\draw[->, >=stealth, very thick] (-0.7,0) -- (5, 0)node[above]{$x$};
	\draw[->, >=stealth, very thick] (0,-0.7) -- (0, 5.5)node[right]{$y$};
	\draw (0, 0)node[below left] {O};
	\coordinate (p1) at (0, 0);
	\coordinate (p2) at (0, 1);
	\coordinate (p3) at (0, 2);
	\coordinate (p4) at (0, 3);
	\coordinate (p5) at (1, 0);
	\coordinate (p6) at (1, 1);
	\coordinate (p7) at (1, 2);
	\coordinate (p8) at (1, 3);
	\coordinate (p9) at (2, 0);
	\coordinate (p10) at (2, 1);
	\coordinate (p11) at (2, 2);
	\coordinate (p12) at (2, 3);
	\coordinate (p13) at (2, 4);
	\coordinate (p14) at (3, 2);
	\coordinate (p15) at (3, 3);
	\coordinate (p16) at (3, 4);
	\coordinate (p17) at (3, 0);
	\coordinate (p18) at (3, 1);
	\coordinate (p19) at (4, 2);
	\coordinate (p20) at (4, 3);
	\coordinate (p21) at (4, 4);
	\coordinate (p22) at (0, 4);
	\coordinate (p23) at (1, 4);
	\coordinate (p24) at (2, 5);
	\coordinate (p25) at (3, 5);
	\draw[->, >=stealth] (p1) -- (0.98, 0); 
	\draw[->, >=stealth] (p2) -- (0.98, 1);
	\draw[->, >=stealth] (p3) -- (0.98, 2);
	\draw[->, >=stealth] (p4) -- (0.98, 3);
	\draw[->, >=stealth] (p5) -- (1.98, 0);
	\draw[->, >=stealth] (p6) -- (1.98, 1);
	\draw[->, >=stealth] (p7) -- (1.98, 2);
	\draw[->, >=stealth] (p8) -- (1.98, 3);
	\draw[->, >=stealth] (p11) -- (2.98, 2);
	\draw[->, >=stealth] (p12) -- (2.98, 3);
	\draw[->, >=stealth] (p13) -- (2.98, 4);
	\draw[->, >=stealth] (p1) -- (0, 0.98); 
	\draw[->, >=stealth] (p2) -- (0, 1.98);
	\draw[->, >=stealth] (p3) -- (0, 2.98);
	\draw[->, >=stealth] (p5) -- (1, 0.98);
	\draw[->, >=stealth] (p6) -- (1, 1.98);
	\draw[->, >=stealth] (p7) -- (1, 2.98);
	\draw[->, >=stealth] (p9) -- (2, 0.98);
	\draw[->, >=stealth] (p10) -- (2, 1.98);
	\draw[->, >=stealth] (p11) -- (2, 2.98);
	\draw[->, >=stealth] (p12) -- (2, 3.98);
	\draw[->, >=stealth] (p14) -- (3, 2.98);
	\draw[->, >=stealth] (p15) -- (3, 3.98);
	\filldraw (p1) circle[radius = 0.5mm]; 
	\filldraw (p2) circle[radius = 0.5mm]; 
	\filldraw (p3) circle[radius = 0.5mm]; 
	\filldraw (p4) circle[radius = 0.5mm]; 
	\filldraw (p5) circle[radius = 0.5mm]; 
	\filldraw (p6) circle[radius = 0.5mm]; 
	\filldraw (p7) circle[radius = 0.5mm]; 
	\filldraw (p8) circle[radius = 0.5mm]; 
	\filldraw (p9) circle[radius = 0.5mm]; 
	\filldraw (p10) circle[radius = 0.5mm]; 
	\filldraw (p11) circle[radius = 0.5mm]; 
	\filldraw (p12) circle[radius = 0.5mm]; 
	\filldraw (p13) circle[radius = 0.5mm]; 
	\filldraw (p14) circle[radius = 0.5mm]; 
	\filldraw (p15) circle[radius = 0.5mm]; 
	\filldraw (p16) circle[radius = 0.5mm]; 
	\filldraw [fill = red, red](p17) circle[radius = 0.5mm]; 
	\filldraw [fill = red, red] (p18) circle[radius = 0.5mm]; 
	\filldraw [fill = red, red] (p19) circle[radius = 0.5mm]; 
	\filldraw [fill = red, red] (p20) circle[radius = 0.5mm]; 
	\filldraw [fill = red, red] (p21) circle[radius = 0.5mm]; 
	\filldraw [fill = red, blue] (p22) circle[radius = 0.5mm]; 
	\filldraw [fill = red, blue] (p23) circle[radius = 0.5mm]; 
	\filldraw [fill = red, blue] (p24) circle[radius = 0.5mm]; 
	\filldraw [fill = red, blue] (p25) circle[radius = 0.5mm]; 
	\draw[domain=2.1:4.2, thick] plot(\x, 3*\x -7)node[above]{$\ell_1$};
	\draw[domain=-0.7:5.1, thick] plot(\x, 0.5*\x +3)node[right]{$\ell_2$};
	\draw (p9) -- (p10) -- (p16) -- (p14) -- (p9); 
	\draw[red] (p17) -- (p18) -- (p21) -- (p19) -- (p17); 
	\draw (p4) -- (p8) -- (p16) -- (p13) -- (p4); 
	\draw[blue] (p22) -- (p23) -- (p25) -- (p24) -- (p22); 
\end{tikzpicture}
\caption{$X_i'$}
\label{koushi3}
\end{center}
\end{minipage}
\end{tabular}
\end{figure}

\subsection{Characterization for $n=2$} 

In the previous subsection, we obtained several properties that 
are satisfied by finite record sets of chip-firing games.

\begin{theorem}\label{main}
Let $\scL$ be a finite language on $\Sigma = \{1, 2\}$. 
Then the following two conditions are equivalent.
	\begin{description}
		\item[(1)] $\scL$ satisfies (LH), (LF), (PM), (abb), and (SP).
		\item[(2)] There exist a finite graph $G$ and an initial 
configuration $\varphi$ of chips such that $\scL=L_{G, \varphi}$. 
	\end{description}
\end{theorem}

\begin{proof}
(2) $\Longrightarrow$ (1) has been already proved in the previous section. 
We shall prove the converse. Assume that the language $\scL$ satisfies 
(LH), (LF), (PM), (abb) and (SP). 
Let us take $X_1$ as in Definition \ref{def:Xi}. Then by (abb), 
the slope of each edge of the polygon $\conv(X_1)$ is at least $1$ 
(Figure \ref{koushi3}). 
We choose a line $\ell_1$ separating $X_1$ and $X_1'$ 
with rational slope $s_1>1$ (we allow $\ell_1$ touches $X_1$ but assume 
$\ell_1\cap X_1'=\emptyset$). 
Note that 
$\begin{pmatrix}
-s_1\\ 1
\end{pmatrix}$ 
is a normal vector. 
Similarly, the slope of each edge of the polygon $\conv(X_2)$ is at most $1$ 
and we can choose a line $\ell_2$ separating $X_2$ and $X_2'$ 
with rational slope $s_2<1$. 
Note that 
$\begin{pmatrix}
s_2\\ -1
\end{pmatrix}$ 
is a normal vector of $\ell_2$. 
We may also assume 
$\ell_1\cap \ell_2=\left\{
\begin{pmatrix}
c_1\\c_2
\end{pmatrix}\right\}$ 
is a rational point. 
Then the set of all words obtained from paths 
on lattice points starting from $0$ in the closed domain 
surrounded by the $x$-axis, $y$-axis, $\ell_1$ and $\ell_2$ 
is equal to $\scL$ because of (LH), (LF) and (PM). The remaining part is to 
realize this language as the record set of chip-firing games on a 
particular graph. 

Let 
\begin{equation}
\label{eq:symm}
A=
\begin{pmatrix}
-s_1s_2&s_2\\
s_2&-1
\end{pmatrix}. 
\end{equation}
and $m>0$ be a positive integer such that all components of $mA$ and $mA
\begin{pmatrix}
c_1\\c_2
\end{pmatrix}$ 
become integers. 
Let us construct a graph with vertices $V=\{1, 2, 3, 4\}$ and 
the following number of edges. 
$
e(1, 2)=ms_2, 
e(1, 3)=m(s_1s_2-s_2), 
e(2, 3)=m(1-s_2), 
e(1, 4)=e(2, 4)=0$ 
and 
$e(3, 4)\gg 0$ (Figure \ref{graph}). 
Then $\deg(1)=ms_1s_2$ and $\deg(2)=m$. 
Consider the initial 
configuration of chips: $\varphi_1=m(s_1s_2c_1-s_2c_2)$, 
$\varphi_2=m(c_2-s_2c_1)$, and $\varphi_3=\varphi_4=0$. 
Then $\scL=L_{G, \varphi}$. 
\end{proof}

\begin{figure}[H]
\begin{center}
\begin{tikzpicture}

\coordinate (P1) at (-1.5, 2.5);
\coordinate (P2) at (1.5, 2.5);
\coordinate (P3) at (0, 1);
\coordinate (P4) at (0, 0);

\draw (P4)  -- node[right] {$\gg 0$} (P3) -- node[right] {$m(1-s_2)$} (P2) -- node[above] {$ms_2$} (P1) -- node[left] {$m(s_1s_2-s_2)$} (P3);

\filldraw [fill = black] (P4) circle[radius = 0.8mm] node[below] {$4$}; 
\filldraw [fill = black] (P3) circle[radius = 0.8mm] node[above] {$3$}; 
\filldraw [fill = black] (P2) circle[radius = 0.8mm] node[right] {$2$}; 
\filldraw [fill = black] (P1) circle[radius = 0.8mm] node[left] {$1$}; 

\draw (P4) .. controls (0.1,0.5) .. (P3); 
\draw (P4) .. controls (-0.1,0.5) .. (P3); 

\draw (P3) .. controls (0.67,1.83) .. (P2); 
\draw (P3) .. controls (0.6,1.9) .. (P2); 

\draw (P1) .. controls (0,2.6) .. (P2); 
\draw (P1) .. controls (0,2.4) .. (P2); 

\draw (P3) .. controls (-0.67,1.83) .. (P1); 
\draw (P3) .. controls (-0.6,1.9) .. (P1); 

\end{tikzpicture}
\caption{The graph realizing the language $\scL$.}
\label{graph}
\end{center}
\end{figure}

\begin{remark}
If the set of fired vertices $\Sigma$ contains at least three vertices, 
the above construction does not work. In particular, 
if $|\Sigma|\geq 3$, in the equation (\ref{eq:symm}), it is not possible 
to choose $A$ to be symmetric in general. It is a challenging problem to 
characterize the language obtained as a finite record set of a chip-firing game 
for $|\Sigma|\geq 3$. 
\end{remark}

\end{document}